\documentclass[10pt, a4paper]{extarticle}

\title{Random points on an algebraic manifold}
\author{Paul Breiding\footnote{Technische Universit\"at Berlin, breiding@math.tu-berlin.de, P.~Breiding has received funding from the European Research Council (ERC) under the European Union's Horizon 2020 research and innovation programme (grant agreement No 787840)} \and Orlando Marigliano\footnote{Max-Planck Institute for Mathematics in the Sciences, Leipzig, orlando.marigliano@mis.mpg.de}}
\date{}

\usepackage{enumitem}
\setlist[enumerate]{leftmargin=-.5in}
\setlist[itemize]{leftmargin=-.5in}

\usepackage[format=plain, font=footnotesize]{caption}
\usepackage[english]{babel}
\usepackage[latin1]{inputenc}		 
\usepackage{graphicx} 				
\usepackage{url}
\usepackage[a4paper]{geometry}
\usepackage{amsmath,amssymb,amsfonts,mathtools,amsthm}
\usepackage{hyperref}
\usepackage{cleveref}
\setenumerate{label=(\arabic*)}
\usepackage{tabularx}
\usepackage[all,2cell]{xy}
\usepackage[sort]{cite}
\usepackage{mathrsfs}
\usepackage{float}

\newcommand{\HC}{\mathbb{C}}

\newcommand{\HP}{\mathbb{P}}

\newcommand{\HR}{\mathbb{R}}
\newcommand{\HS}{\mathbb{S}}

\newcommand{\cE}{\mathcal{E}}

\newcommand{\cH}{\mathcal{H}}
\newcommand{\cI}{\mathcal{I}}

\newcommand{\cK}{\mathcal{K}}
\newcommand{\cL}{\mathcal{L}}
\newcommand{\cM}{\mathcal{M}}
\newcommand{\cN}{\mathcal{N}}

\newcommand{\cS}{\mathcal{S}}

\newcommand{\cU}{\mathcal{U}}
\newcommand{\cV}{\mathcal{V}}
\newcommand{\cW}{\mathcal{W}}

\newcommand{\set}[1]{\left\{#1\right\}}
\newcommand{\cset}[2]{\left\{#1\mid #2\right\}}
\renewcommand{\d}{\mathrm{d}}

\DeclareMathOperator*{\Prob}{\mathrm{Prob}}
\DeclareMathOperator*{\mean}{\mathbb{E}}
\newcommand\restr[2]{\ensuremath{\left.#1\right|_{#2}}}
\newcommand{\deriv}[2]{\mathrm{D}_{#2}#1\,}

\newcommand{\Tang}[2]{\mathrm{T}_{#1} {#2}}

\numberwithin{equation}{section}
\numberwithin{figure}{section}

\theoremstyle{plain}
\newcounter{numbering} \numberwithin{numbering}{section}
\newtheorem{theorem}[numbering]{Theorem}
\newtheorem{lemma}[numbering]{Lemma}
\newtheorem{proposition}[numbering]{Proposition}
\newtheorem{corollary}[numbering]{Corollary}
\theoremstyle{definition}
\newtheorem{definition}[numbering]{Definition}
\theoremstyle{remark}

\newtheorem{remark}[numbering]{Remark}
\theoremstyle{plain}

\crefname{equation}{}{}
\crefname{equation}{}{}
\crefname{figure}{Figure}{Figures}
\crefname{section}{Section}{Sections}
\crefname{lemma}{Lemma}{Lemmata}
\crefname{prop}{Proposition}{Propositions}
\crefname{thm}{Theorem}{Theorems}
\crefname{cor}{Corollary}{Corollaries}
\crefname{dfn}{Definition}{Definitions}
\crefname{notation}{Notations}{Notations}
\crefname{rem}{Remark}{Remarks}
\crefname{claim}{Claim}{claims}
\crefname{equation}{}{}
\crefname{equation}{}{}
\crefname{figure}{Figure}{Figures}
\crefname{section}{Section}{Sections}
\crefname{lemma}{Lemma}{Lemmata}
\crefname{proposition}{Proposition}{Propositions}
\crefname{theorem}{Theorem}{Theorems}
\crefname{corollary}{Corollary}{Corollaries}
\crefname{definition}{Definition}{Definitions}
\crefname{notation}{Notations}{Notations}
\crefname{remark}{Remark}{Remarks}
\crefname{claim}{Claim}{claims}

\begin{document}

\maketitle


\begin{abstract}
Consider the set of solutions to a system of polynomial equations in many variables.
An algebraic manifold is an open submanifold of such a set.
We introduce a new method for computing integrals and sampling from distributions on algebraic manifolds.
This method is based on intersecting with random linear spaces. It produces i.i.d.\ samples,
works in the presence of multiple connected components, and is simple to implement.
We present applications to computational statistical physics and topological data analysis.
\end{abstract}

\section{Introduction}
In statistics and applied mathematics,
\emph{manifolds} are useful models for continuous data.  For example,
in \emph{computational statistical physics} the state space of a collection of particles is a manifold. Each point on this manifold records the positions of all particles in space.
The field of \emph{information geometry} interprets a statistical model as a manifold, each point corresponding to a probability distribution inside the model.
\emph{Topological data analysis} studies the geometric properties of a point cloud in some Euclidean space. Learning the manifold that best explains the position of the points is a research topic in this field.

Regardless of the context, there are two fundamental computational problems associated to a manifold $\mathcal M$.
\begin{enumerate}
\item Approximate the Lebesgue integral $\int_\mathcal M f(x)\d x$ of a given function $f$ on $\mathcal M$.
\item Sample from a probability distribution with a given density on $\mathcal M$.
\end{enumerate}
These problems are closely related in theory. In applications however, they may occur separately.
For instance, in Section~\ref{sec:experiments} we present an example from computational physics that involves only~(1). On the other hand, the subsequent example from topological data analysis is about~(2).

In general these problems are easy for manifolds which admit a differentiable surjection \mbox{$\mathbb R^k\to \mathcal M$}, also called \emph{parametrized} manifolds. They are harder for \emph{non-parametrized} manifolds, which are usually represented as the sets of nonsingular solutions to some system of differentiable \emph{implicit} equations
\begin{equation} \label{defining-equations}
	F_i(x) = 0 \quad (i=1,\dotsc, r).
\end{equation}

The standard techniques to solve (2) in the non-parametrized case involve moving randomly from one sample point on $\mathcal M$ to the next nearby, and fall under the umbrella term of \emph{Markov Chain Monte Carlo (MCMC)} \cite{BG2013, Brubaker2012, CG2011, GHZ2017, Kalos2008, Lelievre2012, LRS2018}.

In this paper, we present a new method that solves (1) and (2) when the functions $F_i$ are \emph{polynomial}.
In simple terms, the method can be described as follows. First, we choose a random linear subspace of complementary dimension and calculate its intersection with $\mathcal M$. Since the implicit equations are polynomial, the intersection can be efficiently determined using numerical polynomial equation solvers \cite{bertini, BT, HOM4PS, Leykin2018, PHCpack}.
The number of intersection points is finite and bounded by the \emph{degree} of the system~\cref{defining-equations}. Next, if we want to solve (1) we evaluate a modified function $\overline f$ at each intersection point, sum its values, and repeat the process to approximate the desired integral. Else if we want to solve (2), after a rejection step we pick one of the intersection points at random to be our sample point. We then repeat the process to obtain more samples of the desired density.

Compared to MCMC sampling, our method has two main advantages. First, we have the option to generate points that are independent of each other. Second, the method is global in the sense that it also works when the manifold has multiple distinct connected components, and does not require picking a starting point $x_0\in \mathcal M$.

The main theoretical result supporting the method is Theorem~\ref{theorem1}. It is in line with a series of classical results commonly known as \emph{Crofton's formulae} or \emph{kinematic formulae}~\cite{santalo} that relate the volume of a manifold to the expected number of its intersection points with random linear spaces. Previous uses of such formulae in applications can be found in~\cite{graph-cuts, crofton-discretized, discretized-improved,crofton-sphere}. Traditionally, one starts by sampling from the set of linear spaces intersecting a ball containing $\mathcal M$. Our contribution is to suggest an alternative sampling method for linear spaces $\{x\in \mathbb R^N|Ax=b\}$, namely by sampling $(A,b)$ from the Gaussian distribution on $\mathbb R^{n\times N}\times \mathbb R^{n}.$ We argue in Section \ref{sec:comparison} that our method is more exact and converges at least as quickly as the above.

Until this point, we have assumed that $\mathcal M$ is the set of nonsingular solutions to an implicit system of polynomial equations~\cref{defining-equations}. In fact, our main theorem also holds for open submanifolds of such a set of solutions. We call them \emph{algebraic manifolds}. Throughout this article we fix an $n$-dimensional algebraic manifold $\mathcal M\subset \mathbb R^N$.

To state our result, we fix a measurable function $f\colon \mathcal M\to \mathbb R_{\geq0}$ with finite integral over $\cM$. We define the auxiliary function $\overline f\colon \HR^{n\times N}\times \HR^n\to \mathbb R$ as follows:
\[
\overline f(A,b)\coloneqq \sum_{x\in \mathcal M: Ax = b} \frac{f(x)}{\alpha(x)} \quad \text{ where } \quad\alpha(x)\coloneqq \frac{\sqrt{1+ \langle x, \Pi_{\mathrm{N}_x \cM}\, x\rangle}}{(1+\Vert x\Vert^2)^\frac{n+1}{2}}\;\frac{\Gamma\left(\frac{n+1}{2}\right)}{\sqrt{\pi}^{\,n+1}}
\]
and where $\Pi_{\mathrm{N}_x \cM}: \HR^N \to \mathrm{N}_x \cM$ denotes the orthogonal projection onto the normal space of $\cM$ at $x\in \cM$.
Note that this projection can be computed from the implicit equations for $\mathcal M$, because $\mathrm{N}_x \cM$ is the row-span of the Jacobian matrix $J(x) = [\frac{\partial F_i}{\partial x_j}(x)]_{1\leq i\leq r, 1\leq j\leq N}$. Therefore, if $Q\in\HR^{N\times r}$ is the Q-factor from the QR-decomposition of $J(x)^T$, then we have $\mathrm{N}_x \cM = QQ^T$. This means that we can easily compute $\overline f(A,b)$ from $\mathcal M\cap \cL_{A,b}$.

The operator $f\mapsto \overline f$ allow us to state the following main result, making our new method precise.

\begin{theorem} \label{theorem1}
Let $\varphi(A,b)$ be the probability density for which the entries of $(A,b)\in\mathbb{R}^{n\times N} \times \HR^n$ are i.i.d.\ standard Gaussian. In the notation introduced above:

(1) The integral of $f$ over $\mathcal M$ is the expected value of $\overline f$:
\[
	\int_\mathcal{M}f(x)\,\d x = \mathbb E_{(A,b)\sim\varphi}\overline{f}(A,b).
\]

(2) Assume that $f:\mathcal M\to \HR$ is nonnegative and that $	\int_\mathcal{M}f(x)\,\d x$ is positive and finite. Let $X\in \mathcal M$ be the random variable obtained by choosing a pair $(A,b)\in \mathbb R^{n\times N}\times \mathbb R^{n}$ with probability
\begin{equation*}
\psi(A,b) := \frac{\varphi(A,b)\,\overline f(A,b)}{\mean_{\varphi} (\overline f)}
\end{equation*}
and choosing one of the finitely many points $X$ of the intersection $\mathcal M \cap \mathcal L_{A,b}$ with probability $f(x) \alpha(x)^{-1}\overline f(A,b)^{-1}$. Then $X$ is distributed according to the scaled density $f(x)/(\int_\mathcal{M}f(x)\,\d x)$ associated to $f(x)$.
\end{theorem}

Using the formula for $\alpha(x)$ we can already evaluate $\overline f(A,b)$ for an integrable function~$f$. Thus we can approximate the integral of $f$ by computing the empirical mean
$$\mathrm{E}(f, k) = \tfrac{1}{k}(\overline{f}_1(A,b)+\cdots+\overline{f}_k(A,b))$$
of a sample drawn from $(A,b)\sim \varphi$. The next lemma yields a bound for the rate of convergence of this approach. It is an application of Chebyshev's inequality and proved in Section \ref{sec:proof_lemma_rate_of_convergence}.

\begin{lemma}\label{lemma_rate_of_convergence}
Assume that $\vert f(x)\vert$ and $\Vert x\Vert$ are bounded on $\cM$. Then,
the variance $\sigma^2(\overline{f})$ of $\overline{f}(A,b)$ is finite and for $\varepsilon>0$ we have
$\Prob\{\vert \mathrm{E}(f,k) -  \int_\mathcal{M}f(x)\,\d x \vert \geq \varepsilon \} \leq \frac{\sigma^2(\overline{f})}{\varepsilon^2 k}.$
\end{lemma}

In \cref{bound_sigma} we provide a deterministic bound for $\sigma^2(\overline{f})$, which involves the degree of the ambient variety of $\cM$ and upper bounds for $\Vert x\Vert$ and $\vert f(x)\vert$ on $\cM$. In our experiments we also use the empirical variance $s^2(\overline{f})$ of a sample for estimating $\sigma^2(\overline{f})$.

For sampling $(A,b)\sim \psi$ in the second part of \cref{theorem1} we could use MCMC sampling. Note that this would employ MCMC sampling for the flat space $\mathbb{R}^{n\times N}\times \HR^n$, which is easier than MCMC for nonlinear spaces like $\cM$. Nevertheless, in this paper we use the simplest method for sampling $\psi$, namely \emph{rejection sampling}. This is used in the experiment section and explained in Section \ref{sampling_psi}.

\subsection{Outline}
In Section \ref{sec:experiments} of this paper, we show applications of this method to examples in topological data analysis and statistical physics. We discuss preliminaries for the proof of the main theorem in Section \ref{sec:prelim} and give the full proof in Section \ref{sec:proof_theorem1}. In Section \ref{sec:proof_lemma_rate_of_convergence} we prove \cref{lemma_rate_of_convergence}. We prove a variant of our theorems for projective algebraic manifolds in Section \ref{sec:projective}. In Section \ref{sec:comparison} we review other methods that make use of a kinematic formula for sampling. Finally, in Section \ref{sec:discussion} we briefly discuss the limitations of our method and possible future work.

\subsection{Acknowledgements}
We would like to thank the following people for helping us with this article:   Diego Cifuentes, Benjamin Gess, Christiane G\"orgen, Tony Lelievre, Mateusz Michalek, Max von Renesse, Bernd Sturmfels and Sascha Timme. We also would like to thank two anonymous referees for valuable comments on the paper.

\subsection{Notation}
The euclidean inner product on $\mathbb R^N$ is $\langle x,y\rangle := x^Ty$ and the associated norm is $\Vert x\Vert := \sqrt{\langle x,x\rangle}$. The unit sphere in~$\HR^N$ is $\HS^{N-1}:=\{x\in \HR^N: \Vert x\Vert =1\}$.
For a function $f:\cM\to\cN$ between manifolds we denote by $\deriv{f}{x}$ the derivative of $f$ at $x\in\cM$. The tangent space of $\cM$ at $x$ is denoted $\mathrm{T}_x\cM$ and the normal space is $\mathrm{N}_x\cM$.

\section{Experiments}\label{sec:experiments}

In this section we apply our main results to examples. All experiments have been performed on macOS 10.14.2 on a computer with Intel Core i5 2,3GHz (two cores) and 8 GB RAM memory. For computing the intersections with linear spaces, we use the numerical polynomial equation solver \texttt{HomotopyContinuation.jl}~\cite{BT}. For plotting we use \texttt{Matplotlib}~\cite{Hunter:2007}. For sampling from the distribution $\psi(A,b)$ we use rejection sampling as described in Section \ref{sampling_psi}.

\begin{figure}[ht]
  \begin{center}
\includegraphics[height = 5cm]{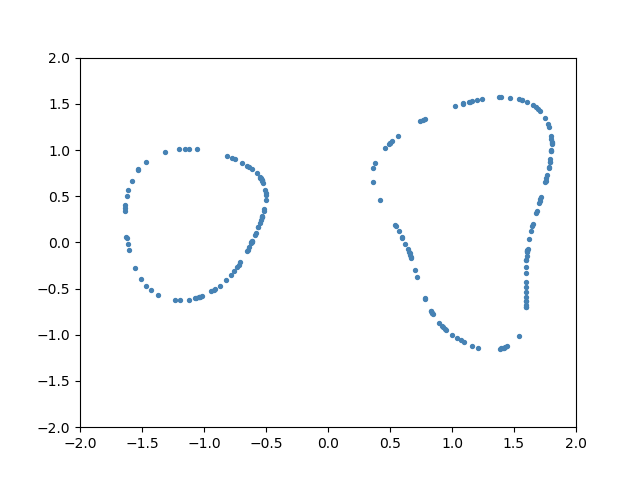}\hfill
\includegraphics[height = 5cm]{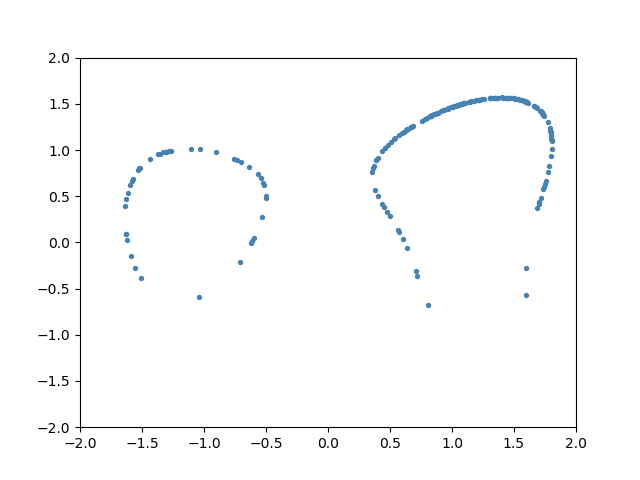}
\end{center}
\caption{Left picture: a sample of 200 points from the uniform distribution on the curve \cref{eq1}. Right picture: a sample of 200 points from the uniform distribution scaled by $e^{2y}$. \label{fig1}}
\end{figure}

As a first simple example, we consider the plane curve $\cM$ given by the equation
\begin{equation}\label{eq1}
 x^4+ y^4 - 3x^2 - xy^2 - y + 1 = 0.
\end{equation}
We have $\mathrm{vol}(\cM) = \mean_\varphi (\overline{1})$. We can therefore estimate the volume of the curve $\cM$ by taking a sample of i.i.d.\ pairs $(A,b)$ and computing the empirical mean $\mathrm{E}(1,k)$ of $\overline{1}$ of the sample. A sample of $k = 10^5$ yields $\mathrm{E}(1,k) = 11.2$. In \cref{lemma_rate_of_convergence} we take $\varepsilon = 0.1$ and the variance of the sample $s^2$, and get an upper bound of $\frac{s^2}{\varepsilon^2k} = 0.008$. Therefore, we expect that $11.2$ is a good approximation of the true volume. We can also take the deterministic upper bound from \cref{bound_sigma} for the variance $\sigma^2$ of $\overline{1}$. Here, we take $\sup_{x\in\cM}\Vert x\Vert = \sqrt{8}$. To get an estimate with accuracy at least $\varepsilon = 0.1$ with probability at least $0.9$ we need a sample of size $k \geq \frac{\sigma^2}{\varepsilon^2 \cdot 0.9} \geq 1421300$. Taking such a sample size we get an estimated volume of  $\approx 11.217$. The code that produced this result is available at \cite{integration}.

Next, we use the second part of Theorem \cref{theorem1} to generate random samples on $\cM$ -- code is available at \cite{sampling}. We show in the left picture of \cref{fig1} a sample of 200 points drawn uniformly from the curve. The right picture shows 200 points drawn from the density given by the normalization of
$
f(x,y) = e^{2y}.
$
As can be seen from the pictures the points drawn from the second distribution concentrate in the upper half of $\mathcal M$, whereas points from the first distribution spread equally around the curve. This experiment also shows how our method generates global samples. The curve has more than one connected component, which is not an obstacle for our method.

Our method is particularly appealing for hypersurfaces like \cref{eq1} because intersecting a hypersurface with a linear space of dimension $1$ reduces to solving a single univariate polynomial equation. This can be done very efficiently, for instance using the algorithm from  \cite{SG2003}, and so for hypersurfaces we can easily generate large sample sets.

\begin{figure}[ht]
  \begin{center}
\includegraphics[height = 5cm]{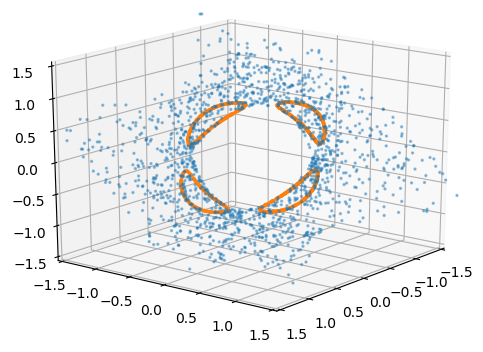}\hfill
\includegraphics[height = 5cm]{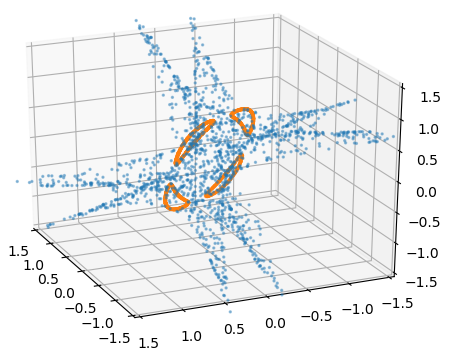}
\end{center}
\caption{The blue points are a sample of 1569 points from the complex Trott curve \cref{trott} seen as a variety in~$\HR^4$ projected to $\HR^3$. The orange points are a sample of 1259 points from the real part of the Trott curve.\label{fig0}}
\end{figure}

The pictures suggest to use sampling for visualization. For instance, we can visualize a semialgebraic piece of the complex and real part of the \emph{Trott curve} $T$, defined by the equation
\begin{equation}\label{trott}
144(x_1^4+x_2^4) - 225(x_1^2+x_2^2) + 350x_1^2x_2^2 + 81 = 0.
\end{equation}
The associated complex variety in $\HC^2$ can be seen as a real variety $T_\HC$ in $\HR^4$. We sample from the real Trott curve $T$ and the complex Trott curve $T_\HC$ intersected with the box $-1.5<\mathrm{Real}(x_1), \mathrm{Imag}(x_1),\mathrm{Real}(x_2), \mathrm{Imag}(x_2)<1.5$. Then, we take a random projection $\HR^4\to\HR^3$ to obtain a sample in $\HR^3$ (the projected sample is \emph{not} uniform on the projected semialgebraic variety). The outcome of this experiment is shown in \cref{fig0}.

\subsection{Application to statistical physics}\label{sec:physics}
In this section we want to apply \cref{theorem1} to study a physical system of $N$ particles $q=(q_1,\ldots,q_N)\in\cM$, where $\cM\subseteq (\mathbb R^3)^N$ is the manifold that models the spacial constraints of the $q_i$. In our example we have $N=6$ and the $q_i$ are the spacial positions of carbon atoms in a \emph{cyclohexane molecule}. The constraints of this molecule are the following algebraic equations:
\begin{equation}\label{cyclohexane_eq}
\cM = \{q=(q_1,\ldots,q_6)\in (\HR^3)^6 \mid
\Vert q_1-q_2\Vert^2 = \cdots= \Vert q_5-q_6\Vert^2 = \Vert q_6-q_1\Vert^2 = c^2\},
\end{equation}
where $c$ is the \emph{bond length} between two neighboring atoms (the vectors $q_i-q_{i+1}$ are called \emph{bonds}). In our example we take $c^2=1$ (unitless). Due to rotational and translational invariance of the equations we define $q_1$ to be the origin, $q_6=(c,0,0)$ and $q_5$ to be rotated, such that its last entry is equal to zero. We thus have $11$ variables.

\begin{figure}[ht]
  \begin{center}
\includegraphics[height = 5.5cm]{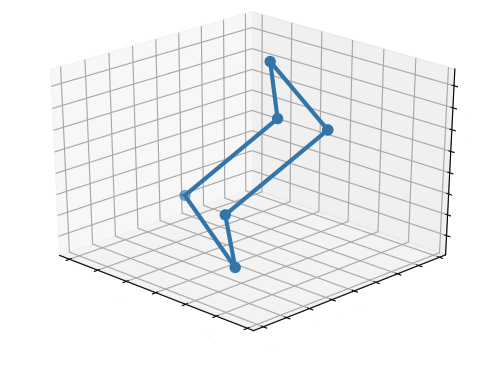}
\end{center}
\caption{\label{fig_physics1}The picture shows a point from the variety \cref{cyclohexane_eq}, for which the angles between two consecutive bonds are all equal to $110.9^{\circ}$ degrees. This configuration is also known as the ``chair'' \cite{BN2011}.}
\end{figure}

Lelievre et.\ al.\  \cite{free_energy} write
``In the framework of statistical physics, macroscopic quantities of interest are written as averages over [...]
probability measures on all the admissible microscopic configurations.'' As the probability measure we take the canonical ensemble \cite{free_energy}. If $E(q)$ denotes the total energy of a configuation $q$, the density in the canonical ensemble is proportional to $f(q) = e^{-E(q)}.$
That is, a configuration is most likely to appear when its energy is minimal. We model the energy of a molecule using an interaction potential, namely the \emph{Lennard Jones potential}
$V(r) =  \frac{1}{4}\,(\frac{c}{r})^{12} -\frac{1}{2}\, (\frac{c}{r})^{6};$
see, e.g., \cite[Equation (1.5)]{free_energy}. Then, the energy function of a system is
$$E(q) =  \sum_{1\leq i<j\leq N} V(\Vert q_i-q_j\Vert).$$
In this example we consider as quantity the average angle between neighboring bonds $q_{i-1} - q_i$ and $q_{i+1} - q_i$
$$\theta(q)=\frac{\angle(q_6 - q_1, q_2-q_1) + \cdots + \angle(q_5-q_6, q_1-q_6)}{6},$$
where $\angle(b_1, b_2):=\arccos \frac{\langle b_1,b_2\rangle }{\Vert b_1\Vert \Vert b_2\Vert}$. We compute the macroscopic state of $\theta(q)$ by determining its distribution
$\mathrm{Prob} \{\theta(q) = \theta_0\}  = \frac{1}{H}\int_{\theta(q) = \theta_0} f(q) \mathrm{d} q,$
where $H=\int_{V} f(q) \mathrm{d} q$ is the normalizing constant. For comparing the probabilities of different values for $\theta$ it suffices to compute
$$\rho(\theta_0)=\int_{\theta(q) = \theta_0} f(q) \mathrm{d} q.$$
We approximate this integral as
$\rho(\theta_0) \approx \frac{\mu_1(\theta_0)}{\mu_2(\theta_0)},$
where
\begin{align*}
\mu_1(\theta_0) &= \int_{\theta(q)>\theta_0 - \Delta\theta\atop \theta(q) < \theta_0 + \Delta\theta} f(q)\; \mathrm{d} q \;\text{ and }
\mu_2(\theta_0) = \int_{\theta(q)>\theta_0 - \Delta\theta\atop \theta(q) < \theta_0 + \Delta\theta} 1 \;\mathrm{d} q
\end{align*}
for some $\Delta \theta >0$ (in our experiment we take $\Delta \theta = 3^\circ$), and we approximate both $\mu_1(\theta_0)$ and $\mu_2(\theta_0)$ for several values of $\theta$ by their empirical means $\mathrm{E}(f,k)$ and $\mathrm{E}(1,k)$, and using \cref{theorem1}. We take $k=10^4$ samples, respectively. The code for this is available at~\cite{cyclo}.

Figure \ref{fig_physics_experiment} shows both the values of the empirical means in the logarithmic scale, and
the ratio of $\mu_1(\theta_0)$ and $\mu_2(\theta_0)$.

\begin{figure}[ht]
  \begin{center}
\includegraphics[height = 5.5cm]{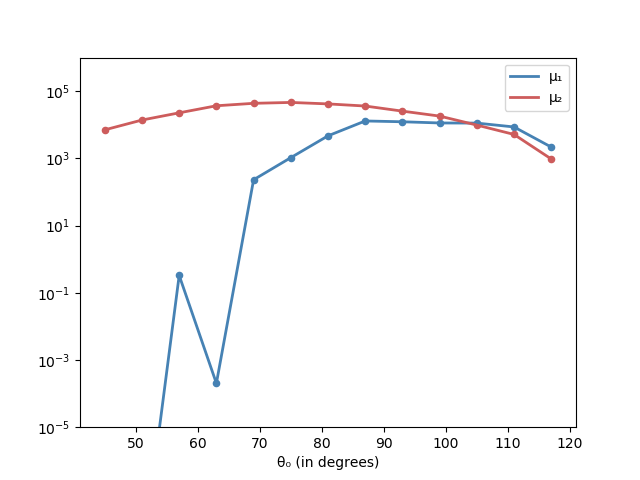}\includegraphics[height = 5.5cm]{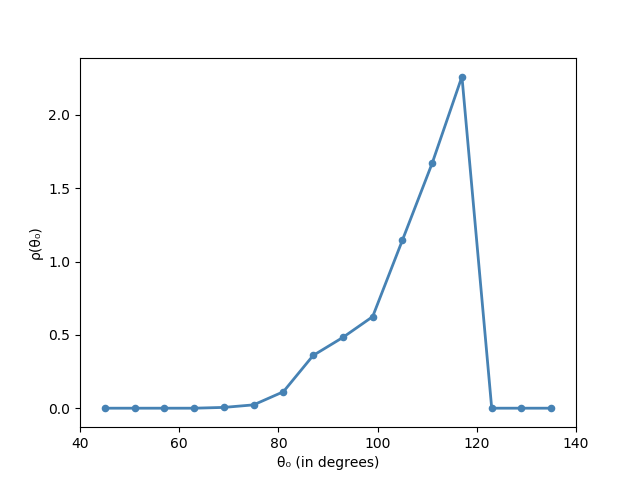}
\end{center}
\caption{\label{fig_physics_experiment}
The left picture shows the approximations of $\mu_1(\theta_0)$ and $\mu_2(\theta_0)$ by the empirical means $\mathrm{E}(f,k)$ and $\mathrm{E}(1,k)$. Both integrals were approximated independently, each by an empirical mean obtained from $10^4$ intersections with linear spaces. The right picture shows the ratio of the empirical means, which approximate $\rho(\theta_0)$.
}
\end{figure}

How good is our estimate? From the plot above we can deduce that $\varepsilon = 10^2$ is a good accuracy for both $\mu_1(\theta_0)$ and $\mu_2(\theta_0)$. Using the variances $s_1^2$ and $s_2^2$ of the samples, respectively, we get
$
\frac{s_1^2}{\varepsilon^2k} \leq 0.002$ and $\frac{s_2^2}{\varepsilon^2k}  \leq 0.002.$
Hence, by \cref{lemma_rate_of_convergence} we expect that the probability that the empirical mean $E(f,k)$ deviates from $\mu_1(\theta_0)$ by more than $\varepsilon$ and that the probability that $E(1,k)$ deviates from $\mu_2(\theta_0)$ by more than $\varepsilon$ are both at most $0.2\%$. We conclude that our approximation of $\rho(\theta) = H\,\mathrm{Prob}\{\theta(q) = \theta_0\}$ is a good approximation.

In fact, \cref{fig_physics_experiment} shows a peak at around $\theta = 114^{\circ}$. It is known that the total energy of the cyclohexane system is minimized when all angles between consecutive bonds achieve $110.9^{\circ}$; see \cite[Chapter 2]{BB2010}. Therefore, our experiment truly gives a good approximation of the molecular geometry of cyclohexane. An example where all the angles between consecutive bonds are $110.9^{\circ}$ is shown in \cref{fig_physics1}.

\subsection{Application to topological data analysis}
We believe that \cref{theorem1} will be useful for researchers working with in topological data analysis using persistent homology (PH). Persistent homology is a tool to estimate the homology groups of a topological space from a finite point sample. The underlying idea is as follows: for varying $t$, put a ball of radius~$t$ around each point and compute the homology of the union of those balls. One then looks at topological features that persists for large intervals in $t$. It is intuitively clear that the point sample should be large enough to capture all of the topological information of its underlying space, and, on the other hand, the sample should be small enough to remain feasible for computations. Dufresne et al.\ \cite{sampling_hauenstein} comment ``Both the theoretical framework for the PH pipeline and its computational costs drive the requirements of a suitable sampling algorithm.'' (For an explanation of the PH pipeline see \cite[Sect. 2]{sampling_hauenstein} and the references therein). They develop an algorithm that takes as input a denseness parameter $\epsilon$ and outputs a sample where each point has at most distance $\epsilon$ to its nearest neighbor. At the same time, their method is trying to keep the sample size as small as possible. In the context of topological data analysis we see our algorithm as an alternative to \cite{sampling_hauenstein}.

In the following we use Theorem \cref{theorem1} for generating samples as input for the PH pipeline from \cite{sampling_hauenstein}. The output of this pipeline is a \emph{persistence diagram}. It shows the appearance and the vanishing of topological features in a 2-dimensional plot. Each point in the plot corresponds to an $i$-dimensional ``hole'', where the $x$-coordinate represents the time $t$ when the hole appears, and the $y$-coordinate is the time when it vanishes. Points that are far from the line $x=y$ should be interpreted as signals coming from the underlying space. The number of those points is used as an estimator for the Betti number $\beta_i$. For computing persistence diagrams we use \texttt{Ripser} \cite{ripser}.

\begin{figure}[ht]
  \begin{center}
\includegraphics[height = 5cm]{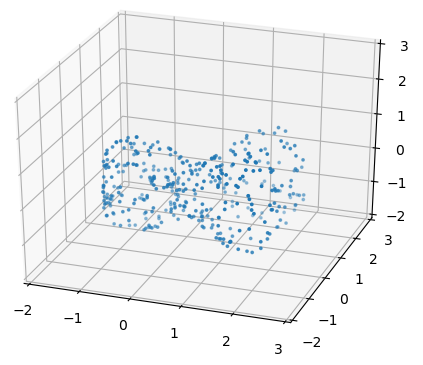}\hfill
\includegraphics[height = 5cm]{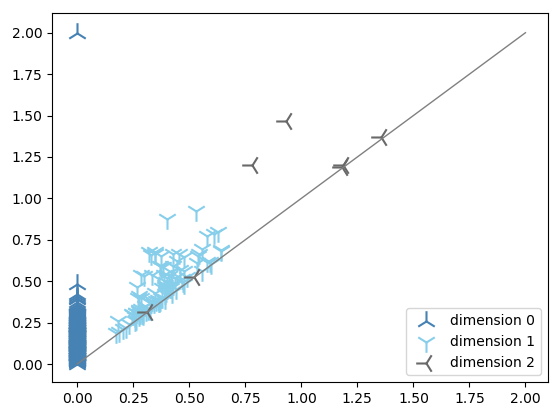}
\end{center}
\caption{The left picture shows a sample of 386 points from the variety \cref{V1}. The right picture shows the corresponding persistence diagram.\label{fig3}}
\end{figure}

First, we consider two toy examples from \cite[Section 5]{sampling_hauenstein}: the surface $\cS_1$ is given by
\begin{equation}\label{V1}
4x_1^4+7x_2^4+3x_3^4-3-8x_1^3+2x_1^2x_2-4x_1^2-8x_1x_2^2-5x_1x_2+8x_1-6x_2^3+8x_2^2+4x_2 = 0.
\end{equation}
\cref{fig3} shows a sample of 386 points from the uniform distribution on $\cS_1$. The associated persistence diagram suggest one connected component, two 1-dimensional and two 2-dimensional holes. The latter two come from two sphere-like features of the variety. The outcome is similar to the diagram from \cite[Figure 6]{sampling_hauenstein}. Considering that the diagram in this reference was computed using~1500 points \cite{parker}, we think that the quality of our diagram is good.

\begin{figure}[ht]
  \begin{center}
\includegraphics[height = 5cm]{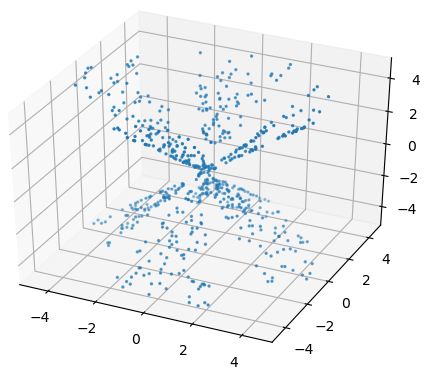}\hfill
\includegraphics[height = 5cm]{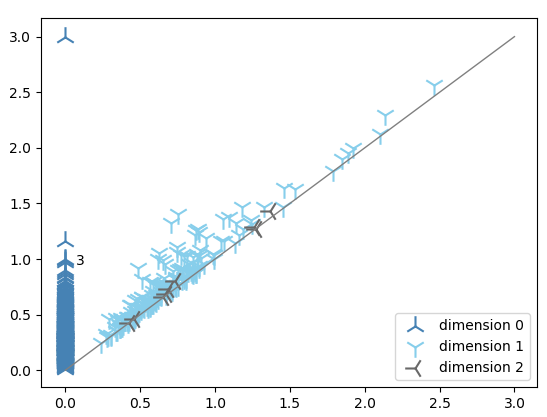}
\end{center}
\caption{The left picture shows a sample of 651 points from the variety $\cS_2$. The right picture shows the corresponding persistence diagram.\label{fig4}}
\end{figure}

The second example is the surface $\cS_2$ given by
{\small
\begin{equation*}
144(x_1^4+x_2^4)-225(x_1^2+x_2^2)x_3^2+350x_1^2x_2^2+81x_3^4+x_1^3+7x_1^2x_2+3(x_1^2+x_1x_2^2)-4x_1-5(x_2^3-x_2^2-x_2)=0.
\end{equation*}
}
\cref{fig3} shows a sample of 651 points from the uniform distribution on $\cS_2$. The persistence diagram on the right suggest one or five connected components. The true answer is five connected components. The diagram from \cite[Figure 6]{sampling_hauenstein} captures the correct homology more clearly, but was generated from a sample of 10000 points \cite{parker}.

\begin{figure}[ht]
  \begin{center}
\includegraphics[height = 5cm]{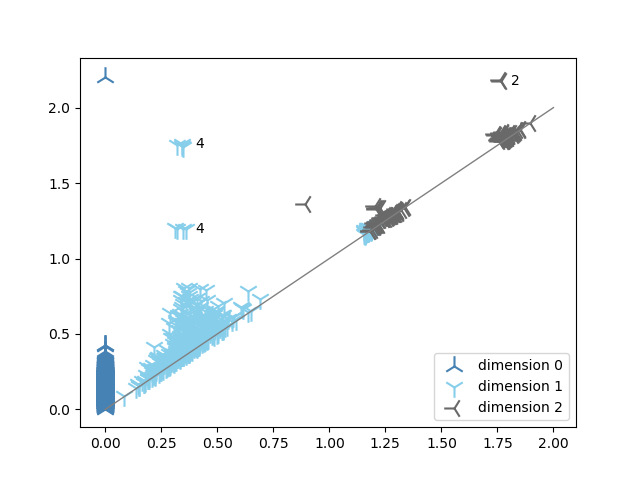}
\end{center}
\caption{The picture shows the persistence diagram of a sample of 1400 points from the variety given by \cref{linkages}.\label{fig5}}
\end{figure}

The next example is from a specific application in kinematics. We quote \cite[Sect. 5.3]{sampling_hauenstein}: ``Consider a regular pentagon in the plane consisting of links with unit length, and with one of the links fixed to lie along the $x$-axis with leftmost point at $(0,0)$. The set of all possible configurations of this regular pentagon is a real algebraic variety.''
The equations of the configuration space are
\begin{equation}\label{linkages}
(x_1+x_2+x_3)^2 + (1+x_4+x_5+x_6)^2 - 1 = 0, \; x_1^2+x_4^2=1, \; x_2^2+x_5^2=1, \; x_3^2+x_6^2=1.
\end{equation}

Here, the zero-th homology is of particular importance because, if the variety is connected, ``the mechanism has one assembly mode which can be continuously deformed to all possible configurations'' \cite{sampling_hauenstein}. \cref{fig5} shows the persistence diagram of a sample of 1400 points from the configuration space. It suggests that the variety indeed has only one connected component. We moreover observe eigth holes of dimension 1 and one or three 2-dimensional holes. The correct Betti numbers are $\beta_0=1,\beta=1=8,\beta_2=1$; see \cite{FS2007}.

\section{Preliminaries}\label{sec:prelim}

In this we first define the degree of a real algebraic variety and explain why the number of intersection points of $\mathcal M$ with a linear space of the right codimension does not exceed the degree of its ambient variety. Then, we recall the coarea formula of integration and discuss some consequences. Finally, we explain how to sample from $\psi(A,b)$ using rejection sampling and we prove an algorithm for sampling $\cL_{A,b}$ in implicit form.

\subsection{Real algebraic varieties} \label{real-algebraic-varieties}
For the purpose of this paper, a (real, affine) \emph{algebraic variety} is a subset $\cV$ of $\mathbb R^N$ such that there exists a set of polynomials $F_1,\dotsc, F_k$ in $N$ variables such that $\cV$ is their set of common zeros. All varieties have a \emph{dimension}  and a \emph{degree}. The dimension of $\cV$ is defined as the dimension of its subspace of non-singular points $\cV_0$, which is a manifold. For the degree we give a definition in the following steps.

 An algebraic variety $\cV$ in $\mathbb R^N$ is \emph{homogeneous} if for all $t\in \mathbb R\setminus \{0\}$ and $x\in \cV$ we have $tx\in \cV$. Homogeneous varieties are precisely the ones where we can choose the $F_i$ above to be homogenous polynomials. Naturally, homogeneous varieties live in the $(N-1)$-dimensional \emph{real projective space} $\mathbb P^{N-1}$. This space is defined as the set $(\mathbb R^{N}\setminus \{0\})/\sim$, where $x\sim y$ if $x$ and $y$ are collinear.
 It comes with a canonical projection map $p\colon (\mathbb R^N\setminus 0) \to \mathbb P^{N-1}$. Then, a \emph{projective variety} is defined as the image of a homogeneous variety $\cV$ under $p$. Its \emph{dimension} is $\dim \cV - 1$.

 Similarly, we define \emph{complex affine, homogeneous, and projective varieties} by replacing $\mathbb R$ with~$\mathbb C$ in the previous definitions. We can pass from real to complex varieties as follows. Let $\cV\subset \mathbb R^N$ be a real affine variety. Its \emph{complexification} $\cV_{\mathbb C}$ is defined as the complex affine variety
\[ \cV_\mathbb{C}\coloneqq\{x\in \mathbb C^n : f(x) = 0 \text{ for all real polynomials $f$
 vanishing on $\cV$} \}.\]
The ``all'' is crucial here. Consider, for instance, the variety in $\HR^2$ defined by $x_1^2+x_2^2 =0$. Obviously, this variety is a single point $\{(0,0)\}$, but $\{x\in \HC^2 : x_1^2+x_2^2 =0 \} = \{(t,\sqrt{-1}\,t) : t\in \HC\}$ is one-dimensional. Nevertheless, the polynomials $x_1=0,x_2=0$ also vanish on $\{(0,0)\}$ and so the complexification of $\cV=\{(0,0)\}$ is $\cV_\HC=\{(0,0)\}$. The following lemma is important.
\begin{lemma}[Lemma 8 in \cite{whitney}] The real dimension of $\cV$ and the complex dimension of its complexification $\cV_\HC$ agree.\label{lemma8whitney}
\end{lemma}
The \emph{Grassmannian} is a smooth algebraic variety $\mathrm{G}(k,\mathbb{C}^{N})$ that parametrizes linear subspaces of~$\mathbb C^{N}$ of dimension $k$. Furthermore, $k$-dimensional affine-linear subspaces of~$\mathbb C^N$ can be seen as $(k+1)$-dimensional linear subspaces of $\mathbb C^{N+1}$ and are parametrized by the \emph{affine Grassmannian} $\mathrm G_{\mathrm{Aff}}(k,\mathbb C^n)$. A \emph{projective linear space} of dimension $k$ is the image of a linear space $\cL\in \mathrm{G}(k+1,\HC^N)$ under the projection $p$. This motivates to define the projective Grassmannian as $\mathrm{G}(k,\mathbb{P}^{N-1})\coloneqq \{p(\cL) : \cL \in \mathrm{G}(k+1,\mathbb{C}^{N})\}$.

Now, we have gathered all the material to give a precise definition of the degree: let $\cV\subset \mathbb P^{N-1}_\mathbb{C}$ be a complex projective variety of dimension $n$. There exists a unique natural number $d$ and a lower-dimensional subvariety $\cW$ of $\mathrm{G}(N-n,\mathbb{P}^{N-1})$ with the property that for all linear spaces $\cL \in \mathrm{G}(N-n,\mathbb{P}^{N-1})\backslash \cW$ the intersection $\cV\cap \cL$ consists of $d$ distinct points \cite[Sect.~18]{Harris1992}. Furthermore, the number of such intersection points only decreases when $\cL \in \cW$. This number~$d$ is called the \emph{degree} of the projective variety $\cV$.
The \emph{degree} of a complex affine variety $\cV\subset \mathbb C^N$ is defined as the degree of the smallest projective variety containing the image of $\cV$ under the embedding $\mathbb C^N\hookrightarrow \mathbb P_\mathbb{C}^{N}$ sending $x$ to $p([1,x])$.

The definition of degree of complex varieties is standard in algebraic geometry. In this article, however, we are solely dealing with real varieties. We therefore make the following definition, which is not standard in the literature, but which fits in our setting.
\begin{definition}
The \emph{degree} of a real affine variety $\cV$ is the degree of its complexification.
The \emph{degree} of a real projective variety $\cV$ is the degree of the image of the complexification of $p^{-1}(\cV)$ under $p$.
\end{definition}

Using \cref{lemma8whitney} we make the following conclusions, after passing from the Grassmannians $\mathrm{G}_{\mathrm{Aff}}(N-n,\mathbb R^N)$ and $\mathrm{G}(N-n,\mathbb R ^{N})$ to the parameter spaces $\mathbb R^{n\times N} \times \mathbb R^n$ and $\mathbb R^{n\times N}$.
\begin{lemma}\label{lem_intersection}
Let $\cV\subset \HR^N$ be an affine variety of dimension $n$ and degree $d$.
Except for a lower-dimensional subset of $\mathbb R^{n\times N}\times \mathbb R^N$, all affine linear $\cL_{A,b}=\{x\in\HR^{N} : Ax = b\}\subset \mathbb R^{N}$ intersect~$\cV$ in at most $d$ many points.

Let $\cV\subset \HP^{N-1}$ be a projective variety of dimension $n$ and degree $d$.
Except for a lower-dimensional subset of $\mathbb R^{n\times N}$, all linear subspaces $\cL_A=\{x\in\HP^{N-1} : Ax = 0\}\subset \mathbb{P}^{N-1}$ intersect $\cV$ in at most $d$ many points.
\end{lemma}

\subsection{The coarea formula}
The \emph{coarea formula of integration} is a key ingredient in the proof of the main Theorems. This formula says how integrals transform under smooth maps. A well known special case is integration by substitution. The coarea formula generalizes this from integrals defined on the real line to integrals defined on differentiable manifolds.

Let $\cM,\cN$ be Riemannian manifolds and $\d v$, $\d w$ be the respective volume forms. Furthermore, let $h:\cM\to \cN$ be a smooth map. A point $v\in \cM$ is called a \emph{regular point of $h$} if $\deriv{h}{v}$ is surjective. Note that a necessary condition for regular points to exist is $\dim \cM \geq \dim \cN$.

For any $v\in \cM$ the Riemannian metric on $\cM$ defines orthogonality on~$\Tang{v}{\cM}$. For a regular point~$v$ of $h$ this implies that the restriction of $\deriv{h}{v}$ to the orthogonal complement of its kernel is a linear isomorphism. The absolute value of the determinant of that isomorphism is the \emph{normal Jacobian of~$h$ at $v$}. Let us summarize this in a definition.
\begin{definition}
Let $h:\cM\to \cN$ be a smooth map and $v\in \cM$ be a regular point of $h$. Let~$(\,\cdot\,)^\perp$ denote the orthogonal complement. The normal Jacobian of $h$ at $v$ is defined
$$\mathrm{NJ}(h,v) := \left\vert\det\left(\restr{\deriv{h}{v}}{(\ker \deriv{h}{v})^\perp}\right)\right\vert.$$
\end{definition}
We also need the following theorem (see, e.g., \cite[Theorem A.9]{condition}).
\begin{theorem}\label{A9}
Let $\cM,\cN$ be smooth manifolds with $\dim \cM \geq \dim \cN$ and let $h : \cM \to \cN$ be a smooth map. Let $w\in \cN$ be such that all $v\in h^{-1}(w)$ are regular points of $h$. Then, the fiber $h^{-1}(w)$ over $w$ is a smooth submanifold of $\cM$ of dimension $\dim \cM - \dim \cN$ and the tangent space of~$h^{-1}(w)$ at $v$ is $\Tang{v}{h^{-1}(w)} = \ker \deriv{h}{v}$.
\end{theorem}
A point $w\in \cN$ satisfying the properties in the previous theorem is called \emph{regular value of $h$}.
By \emph{Sard's lemma} the set of all $w\in \cN$ that are not a regular value of $h$ is a set of measure zero. We are now equipped with all we need to state the coarea formula. See \cite[(A-2)]{howard} for a proof.
\begin{theorem}[The coarea formula of integration]\label{coarea}
Suppose that $\cM, \cN$ are Riemannian manifolds, and let $h:\cM\to \cN$ be a surjective smooth map. Then we have for any function $a: \cM \to \HR$ that is integrable with respect to the volume measure of $\cM$ that
$$\int_{v\in \cM} a(v) \;\d v  = \int_{w\in \cN} \left(\int_{u\in h^{-1}(w)}\, \frac{a(u)}{\mathrm{NJ}(h,u)}\;\d u \right) \,\d w,$$
where $\d u$ is the volume form on the submanifold $h^{-1}(w)$.
\end{theorem}
The following corollary from the coarea formula is important.
\begin{corollary}\label{coarea-densities}
Let $h\colon \cM\to\cN$ be a smooth surjective map of Riemannian manifolds.

(1) Let $X$ be a random variable on $\cM$ with density $\beta$. Then $h(X)$ is a random variable on $\cN$ with density
\[
	\gamma(y) = \int_{x\in h^{-1}(y)} \frac{\beta(x)}{\mathrm{NJ}(h,x)} \,\d x.
\]

(2) Let $\psi$ be a density on $\cN$ and for all $y\in \cN$, let $\rho_y$ be a density on $h^{-1}(y)$. The random variable $X$ on $\cM$ obtained by independently taking $Y\in \cN$ with density $\psi$ and $X\in h^{-1}(Y)$ with density $\rho_Y$ has density
 \[
 	\beta(x) = \psi(h(x))\rho_{h(x)}(x)\mathrm{NJ}(h,x).
 \]
\end{corollary}
\begin{proof}
The first part follows directly from the coarea formula. For the second part,
it suffices to note that for measurable $\cU\subset \cM$ we have
\begin{align*}
	\int_{y\in h(\cU)}\int_{x\in h^{-1}(y)\cap \cU}\psi(h(x))\rho_{h(x)}(x)\,\d x \d y
	&= \int_{y\in h(\cU)}\int_{x\in h^{-1}(y)\cap \cU}\\ \frac{\beta(x)}{\mathrm{NJ}(h,x)}\,\d x \,\d y
	&= \int_{x\in\cU}\beta(x)\,\d x;
\end{align*}
see also \cite[Remark 17.11]{condition}.
\end{proof}

\subsection{Sampling from the density on affine-linear subspaces}\label{sampling_psi}
Our method for sampling from an algebraic manifold involves taking a distribution $\varphi$ on the parameter space $\mathbb R^{n\times N} \times \mathbb R^n$ of hyperplanes of the right dimension which is easy to sample, and turning it into another density $\psi$. In this section, we explain how to sample from $\psi$ with rejection sampling. In the following we denote elements of $\mathbb R^{n\times N} \times \mathbb R^n$ by $(A,b)$.

\begin{proposition}\label{sample_psi}
Let $\kappa$ be any number satisfying
$0< \kappa\cdot \sup_{(A,b)
}\overline f(A,b) \leq  1.$
Consider the binary random $Z\in\set{0,1}$ with
  $\Prob \{Z = 1 \mid (A,b)\} = \kappa\,\overline f(A,b).$
Then, $\psi$ is the density of the conditional random variable $((A,b)\mid Z=1)$.
\end{proposition}
\begin{proof}
We denote the density of the conditional random variable $((A,b)\mid Z=1)$ by $\lambda$. Bayes' Theorem implies
$\lambda(A,b)\,\Prob\{Z=1\} = \Prob \{Z = 1 \mid (A,b)\} \, \varphi(A,b),$
which, by assumption, is equivalent to
$$ \lambda(A,b) = \frac{\kappa\, \overline f(A,b)\,\varphi(A,b)}{\Prob\{Z=1\}} .$$
By the definition of $Z$ we have
  $\Prob\{Z=1\}  = \kappa \,\mean_{(A,b)\sim\varphi}\,\overline f(A,b)$.
Hence, $$\lambda(A,b) = \frac{\overline f(A,b)\,\varphi(A,b)}{\mean_{(A,b)\sim\varphi}\,\overline f(A,b)} = \psi(A,b).$$
This finishes the proof.
\end{proof}
\cref{sample_psi} shows that $\psi$ is the density of a conditional distribution. A way to sample from such distributions is by \emph{rejection sampling}: for sampling $((A,b) \mid Z=1)$ we may sample from the joint distribution $(A,b,Z)$ and then keep only the points with $Z=1$. The strong law of large numbers implies the correctness of rejection sampling. Indeed, if $(A_i,b_i,Z_i)$ is a sequence of i.i.d.\ copies of $(A,b,Z)$ and $\mathcal U$ is a measurable set with respect to the Lebesgue measure on $\mathbb R^{n\times N}\times \mathbb R^{n}$, then we have
\begin{align*}
	\frac{\#\{i\mid (A_i,b_i)\in \cU, Z_i = z, i\leq n\}}{\#\{i\mid Z_i = z, i\leq n\}}
	&= \frac{
		{\frac{1}{n}\,\#\{i\mid (A_i,b_i)\in \cU, Z_i = z, i\leq n\}}
	} {
		{\frac{1}{n}\,\#\{i\mid Z_i = z, i\leq n\}}
	}
	\\ &\xrightarrow{\text{a.s.\footnotemark}}
	\frac{\Prob\{(A,b)\in \cU, Z=z\}}{\Prob\{Z=z\}}
	\\ &= \Prob_{(A,b)|Z=z}(\cU).
\end{align*}

For sampling $Z$, however, we must compute a suitable $\kappa$. This can be done as follows. Let $d$ be the degree of the ambient variety of $\mathcal M$. We assume we know upper bounds $K$ for $f(x)$ and $C$ for $\Vert x\Vert^2$, both as $x$ ranges over $\mathcal M$, and set
\[
 \kappa = \frac{1}{dK} \frac{\Gamma(\frac{n+1}{2})}{\sqrt{\pi}^{\, n+1}}\frac{1}{(1+C)^\frac{n+1}{2}}.
\]
Then we have $0 < \kappa \overline f(A,b)\leq 1$ for all $(A,b)$ as needed. With $\kappa$, we have everything we need to carry out the sampling method.

How to obtain the upper bounds $K$ and $C$? For $K$, we might just know the maximum of~$f$. For example, if we want to sample from the uniform distribution, then we may use $f = 1$. In more complicated cases, we could approximate $\max f$ by repeatedly sampling $(A,b)\sim \varphi$ and recording the highest value $f$ takes on the points in the intersection $\mathcal M\cap \mathcal L_{(A,b)}$. Casella and Robert \cite{MCSM} call this approach \emph{stochastic exploration}.

We might know $C$ a priori, for example because we restrict the manifold $\mathcal M$ to a box in $\mathbb R^N$. We can also restrict the manifold to a box after determining by sampling what the size of the box should be. We could also estimate $\max \Vert x\Vert^2$ by sampling as for $\max f$. Sometimes we can also use
\emph{semidefinite-programming} \cite{sos} to bound polynomial functions like~$\Vert x\Vert^2$ on a variety. Note that the probability for rejection increases as $C$ increases. We thus seek to find a $C$ which is as small as possible. If our given function $f$ is invariant under translation, we may translate $\cM$ to decrease $C$. For instance, sampling from the uniform distribution on the circle $(x_1-100)^2 + (x_2-100)^2 = 1$ is the same as sampling on $x_1^2+x_1^2=1$ and then translating by adding $(100,100)$ to each sample point. The difference between the two is that for the first variety we need $C=101$, whereas for the second we can use $C=1$.

\footnotetext{almost surely}

\subsection{Sampling linear spaces in explicit form}

Sometimes it is useful to sample the linear space $\cL_{A,b}$ in explicit form, and not in implicit form $Ax=b$. For instance, if $\cV$ is a hypersurface given by an equation $F(x)=0$, then intersecting $\cV$ with a line $u+tv$ can be done by solving the univariate equation $F(u+tv)=0$. The next lemma shows how to pass from implicit to explicit representation in the Gaussian case.
\begin{lemma}
Let $\cL_{A,b}=\{x\in\HR^N \mid Ax=b\}$ be a random affine linear space given by i.i.d.\ standard Gaussian entries for $A\in\HR^{n\times N}$ and $b\in\HR^{n}$. Consider another random linear space
$$\cK_{u,v_1,\ldots, v_{N-m}} =\{ u+t_1v_1+\cdots t_{N-n}v_{N-n} \mid t_1,\ldots,t_{N-n}\in\HR\},$$
where $u,v_1,\ldots, v_{N-m}$ are obtained as follows. Sample a matrix $U\in\HR^{(N-n+1)\times (N+1)}$ with i.i.d.\ standard Gaussian entries, and let\enlargethispage{\baselineskip}
$$\begin{pmatrix} u \\ 1\end{pmatrix},\;   \begin{pmatrix} v_1 \\ 0\end{pmatrix},\;\ldots,\;\begin{pmatrix} v_{N-n} \\ 1\end{pmatrix}\in\mathrm{rowspan}(U).$$
Then, we have $\cK_{u,v_1,\ldots, v_{N-m}}\sim \cL_{A,b}$.
\end{lemma}
\begin{proof}
Consider the linear space $\widetilde\cL_{A,b}:=\{z\in\HR^{N+1} \mid [A, -b]z=0\}.$ This is a random linear space in the Grassmannian $\mathrm{G}(N+1-n,\mathbb R^{N+1})$. The affine linear space is given as
$\cL_{A,b}:=\{u+t_1v_1+\cdots t_{N-n}v_{N-n}\},$ where
$$\begin{pmatrix} u \\ 1\end{pmatrix},\;   \begin{pmatrix} v_1 \\ 0\end{pmatrix},\;\ldots,\;\begin{pmatrix} v_{N-n} \\ 1\end{pmatrix}\in\mathrm{ker}([A, -b]).$$
Now the kernel of $[A, -b]$ is a random linear space in $\mathrm{G}(n,\mathbb R^{N+1})$, which is invariant under orthogonal transformations. By \cite{Leichtweiss} there is unique orthogonally invariant probability distribution
on the Grassmannian $\mathrm{G}(n, \HR^{N+1})$. Since $\mathrm{rowspan}(U)$ is also orthogonally invariant, we find that $\mathrm{rowspan}(U)\sim \mathrm{ker}([A, -b])$, which concludes the proof.
\end{proof}

\section{Proof of Theorem~\ref{theorem1}}\label{sec:proof_theorem1}

We begin by giving an alternate description of the function $\alpha(x)$.

\begin{lemma}\label{alpha_eq}
$
\alpha(x) = \int_{A\in \mathbb R^{n\times N}}
\varphi(A,Ax)|\det(A|_{\Tang{x}{\cM}})|\,\d A.
$
\end{lemma}
\begin{proof}
Let $\alpha'(x)$ be the right hand side of the formula. Let $U\in O(N)$ be an orthogonal matrix such that $Ux = (0,\dotsc,0,x_N)^T$ and consider the manifold $\mathcal N = U\cdot \mathcal M$. We have $\mathrm T_{Ux}\mathcal N = U \mathrm T_x \mathcal M$ and $\det(\restr{A}{\Tang{x}{\mathcal M}}) =
\det(\restr{AU^T}{\Tang{Ux}{\mathcal N}})$. After the change of variables $A\mapsto AU^T$ we get
\[
\alpha'(x) =
\int_{A\in\HR^{n\times N}}\, \left\vert\det(\restr{A}{\Tang{Ux}{\mathcal N}})\right\vert\;\varphi(A,AUx)\, \d A.
\]
By definition of the Gaussian density, we have $\varphi(A,AUx)=\frac{1}{(\sqrt{2\pi})^n}\phi(AR)$,
where $\phi$ is the Gaussian density on $\mathbb R^{n\times N}$ and $R=\operatorname{diag}(1,\dotsc,1,\sqrt{1+x_N^2}) \in \mathbb R^{N\times N}$.
Let us write $B = AR$. A change of variables from $A$ to $B$ yields
$$\alpha'(x)=
\frac{1}{(1+x_N^2)^\frac{n}{2}\,(\sqrt{2\pi})^{n}}
\int_{B\in\HR^{n\times N}}\,
\left\vert\det(\restr{BR^{-1}}{\Tang{Ux}{\cN}})\right\vert\;
\phi(B)\, \d B.$$
Let $W\in \mathbb{R}^{N\times n}$ be a matrix whose columns form an orthonormal basis for $\Tang{Ux}{\cN}$ and write $M:= R^{-1}W$. Then we have $\det(\restr{BR^{-1}}{\Tang{Ux}{\cN}}) = \det(BM)$ and so
$$\alpha'(x)= \frac{1}{(1+x_N^2)^\frac{n}{2}\,\sqrt{2\pi}^{\,n}}\,\mean_{B\sim \phi} \left\vert\det(BM)\right\vert.$$
We now write $\mean_{B\sim \phi} \left\vert\det(BM)\right\vert = \mean_{B\sim \phi} \det(M^TB^TBM)^\frac{1}{2}$. By \cite[Theorem 3.2.5]{muirhead} the matrix $C:=M^TB^TBM \in \HR^{n\times n}$ is a \emph{Wishart matrix} with covariance matrix $M^TM$.
By \cite[Theorem 3.2.15]{muirhead}, we have $\mean\det(C)^\frac{1}{2} = \det(M^TM)^\frac{1}{2} \tfrac{1}{\sqrt{\pi}} \sqrt{2}^{\,n}\Gamma(\tfrac{n+1}{2})$. Altogether, this shows that
$$\alpha'(x)= \frac{\det(M^TM)^\frac{1}{2} }{(1+x_N^2)^\frac{n}{2}}\,\frac{\Gamma\left(\frac{n+1}{2}\right)}{\sqrt{\pi}^{\,n+1}}.$$
Moreover, we have
\begin{align*}
M^TM &= W^TR^{-T}R^{-1}W\\
 &= W^T\operatorname{diag}(1,\dotsc,1,\tfrac{1}{1+x_N^2})W
 \\ &= \mathbf 1 - \tfrac{1}{1+x_N^2}W\operatorname{diag}(0,\dotsc,0,x_N^2)W.
\end{align*}
The second summand in the last expression is a rank-one matrix with the single non-zero eigenvalue $-\tfrac{||W^TUx||^2}{1+||Ux||^2}$. Taking determinants we get
\[
\det(M^TM) = 1 - \frac{||W^TUx||^2}{1+||Ux||^2}=
\frac{1 + ||Ux||^2 - ||W^TUx||^2}{1+||Ux||^2} =
\frac{1 + ||x||^2 -||\Pi_{\Tang{x}{\cM}}x||^2}{1+||x||^2},
\]
where $\Pi_{\Tang{x}{\cM}}$ denotes the orthogonal projection onto the tangent space. Since we have $||x||^2 -||\Pi_{\Tang{x}{\cM}}x||^2 = \langle x, \Pi_{\mathrm{N}_x \cM}\, x\rangle$, this implies
$$\alpha'(x)= \frac{\sqrt{1+ \langle x, \Pi_{\mathrm{N}_x \cM}\, x\rangle}}{(1+x_N^2)^\frac{n+1}{2}\,}\,\frac{\Gamma\left(\frac{n+1}{2}\right)}{\sqrt{\pi}^{\,n+1}} = \alpha(x).$$
This concludes the proof.
\end{proof}

We are now prepared to prove Theorem~\ref{theorem1}.
\begin{proof}[Proof of Theorem~\ref{theorem1}]
We first prove the first part. The support of~$\overline{f}(A,b)$ is a full dimensional subset of $\HR^{n\times N}\times \HR^n$ and it is contained in the complement of the set of all $(A,b)$ for which $\cM\cap\cL_{A,b}=\emptyset$. We let $\mathcal X$ denote the interior of the support of $\overline{f}(A,b)$, so that $\mean_{(A,b)\sim \varphi} \,\overline{f}(A,b) = \int_{\mathcal X} \overline{f}(A,b) \varphi(A,b)\d (A,b)$.

Let $\pi\colon \mathbb R^{n\times N} \times \mathcal M\to  \mathcal X$ be the map sending a pair $(A,x)$ to $(A,Ax)$. We have $\deriv{\pi}{(A,x)}(\dot A,\dot x)=(\dot A,\dot A x + A \dot x)$,
so the derivative of $\pi$ can be identified with the matrix~$\left(\begin{smallmatrix}\mathbf 1 & 0 \\ * & A\end{smallmatrix}\right)$. This shows that $\mathrm{NJ}(\pi,(A,x))=|\det(A|_{\Tang{x}{\cM}})|$. Therefore, by \cref{coarea},
$$\mean_{(A,b)\sim \varphi} \,\overline{f}(A,b)= \int_{\mathbb R^{n\times N} \times \mathcal M}\, \frac{f(x)}{\alpha(x)} \,|\det(A|_{\Tang{x}{\cM}})| \,\varphi(A,Ax)\; \d (A,x).$$
The projection $\mathbb R^{n\times N}\times \cM\to \cM$ on the second factor has normal Jacobian one everywhere. Applying \cref{coarea} again yields
$$\mean_{(A,b)\sim \varphi} \,\overline{f}(A,b)= \int_{\cM} \frac{f(x)}{\alpha(x)} \left(\int_{\mathbb R^{n\times N}}\,  \,|\det(A|_{\Tang{x}{\cM}})| \,\varphi(A,Ax)\; \d A \right)\d x =  \int_{\cM} f(x) \d x,$$
the second inequality by \cref{alpha_eq}. This proves the first part.

Now, we prove the second part, where we assume that $f\colon \mathcal M\to \mathbb R_{>0}$ is nonnegative. Recall that
$\psi(A,b) = \frac{\varphi(A,b)\overline f(A,b)}{\mean_{\varphi}(\overline f)}$. Since $\mean_{\varphi}(\overline f) = \int_{\cM} f(x) \d x$ is positive and finite by the first part of the theorem, we find that $\psi$ is a well defined probability density. The support of $\psi$ is contained in the closure of $\mathcal X$ and therefore $\cM\cap\cL_{A,b}$ is almost surely non-empty and finite.

Let $Y=(A,x)\in \mathbb R^{n\times N}\times \mathcal M$
be the random variable defined by first choosing $(A,b)\sim \psi$ and then taking $x\in \mathcal M \cap \mathcal L_{A,b}$ with probability $f(x)\alpha(x)^{-1}\overline f(A,b)^{-1}$. By construction, $\pi(Y)\sim \psi$. We use \cref{coarea-densities} (2) and find that $Y$ has density
\[
\beta(A,x)
= \frac{\psi(A,Ax)f(x)\mathrm{NJ}(\pi,(A,x))}
{\alpha(x)\overline f(A,Ax)}.
\]
Recall that  $\mathrm{NJ}(\pi,(A,x))=|\det(A|_{\Tang{x}{\cM}})|$ and that the projection $\mathbb R^{n\times N}\times \cM\to \cM$ on the second factor has normal Jacobian one everywhere. Therefore, by \cref{coarea-densities} (1), the random point $x\in \cM$ has density $\gamma$ with
\begin{align*}
\gamma(x) &= \int_{A\in \mathbb R^{n\times N}}\beta(A,x)\,\d A
\\  &= \frac{f(x)}{\alpha(x)}\int_{A\in \mathbb R^{n\times N}}
\frac{\psi(A,Ax)|\det(A|_{\Tang{x}{\cM}})|}{\overline f(A,Ax)}\,\d A
\\ &= \frac{f(x)}{\alpha(x)\mathbb E_\varphi(\overline f)}\int_{A\in \mathbb R^{n\times N}}
\varphi(A,Ax)|\det(A|_{\Tang{x}{\cM}})|\,\d A.
\end{align*}
Using \cref{alpha_eq} yields
$\gamma(x)= \frac{f(x)}{\mathbb E_\varphi(\overline f)}$. This finishes the proof of Theorem \ref{theorem1}.
\end{proof}

\section{Proof of Lemma \ref{lemma_rate_of_convergence}}\label{sec:proof_lemma_rate_of_convergence}
First, we can bound
$\alpha(x)\geq  \frac{1}{1+\sup_{x\in \mathcal M} \Vert x\Vert^2}\;\frac{\Gamma\left(\frac{n+1}{2}\right)}{\sqrt{\pi}^{\,n+1}}.$
Let $d$ be the degree of the ambient variety of $\mathcal M$. With probability one $\cM\cap \cL_{A,b}$ consists of at most $d$ points and so we have
\begin{align*}
\mean_{(A,b)\sim \varphi} \overline{f}(A,b)^2 &= \mean_{(A,b)\sim \varphi} \left(\sum_{x\in\mathcal M\cap \cL_{A,b}} \frac{f(x)}{\alpha(x)}\right)^2\\
&\leq  \mean_{(A,b)\sim \varphi} \left(\sum_{x\in\mathcal M\cap \cL_{A,b}}  \frac{\vert f(x)\vert}{\alpha(x)}\right)^2\\
&\leq  \mean_{(A,b)\sim \varphi} \left(\sum_{x\in\mathcal M\cap \cL_{A,b}}  \frac{\sup_{x\in\cM}\vert f(x)\vert}{\inf_{x\in\cM}\alpha(x)}\right)^2\\
&\leq d^2(1+\sup_{x\in \mathcal M} \Vert x\Vert^2)^{n+1} \frac{\pi^{n+1}}{\Gamma\left(\frac{n+1}{2}\right)^2}\, \sup_{x\in \mathcal M} f(x)^2.
\end{align*}
We also have $	\sigma^2(\overline f) \leq \mean_{(A,b)\sim \varphi} \overline{f}(A,b)^2$, and therefore
\begin{equation}\label{bound_sigma}
	\sigma^2(\overline f) \leq d^2(1+\sup_{x\in \mathcal M} \Vert x\Vert^2)^{n+1} \frac{\pi^{n+1}}{\Gamma\left(\frac{n+1}{2}\right)^2}\, \sup_{x\in \mathcal M} f(x)^2
\end{equation}
is finite. We may therefor use Chebyshev's inequality to deduce that
\begin{equation}\label{chebyshev}
\Prob\left\{\left\vert \mathrm{E}(f,k) -  \int_\mathcal{M}f(x)\,\d x \right\vert \geq \varepsilon \right\} \leq \frac{\sigma^2(\overline f)}{\varepsilon^2k}.
\end{equation}
This finishes the proof.\qed

\section{Sampling from projective manifolds}\label{sec:projective}

In this section we prove variations of Theorems~\ref{theorem1} for projective algebraic manifolds.

Real projective space $\HP^{N-1}$ from \cref{sec:prelim} is a compact Riemannian manifold with a canonical metric, the \emph{Fubini-Study} metric. Namely, let $p:\HR^{N}\backslash \{0\}\to \HP^{N-1}$ be the canonical projection. Restricted to the unit sphere $\HS^{N-1}$, the projection $p$ identifies antipodal points.
We define a subset $\cU\subset \HP^{N-1}$ to be open if and only if $\restr{p}{\HS^{N-1}}^{-1}(\cU)$ is open. This gives $\HP^{N-1}$ the structure of a differential manifold. The Riemannian structure on $\HP^{N-1}$ is defined as $\langle \dot a,\dot b\rangle := \langle \deriv{p}{x}^{-1}\dot a,\deriv{p}{x}^{-1}\dot b\rangle $ for $\dot a,\dot b\in\Tang{x}\HP^{N-1}$.
This metric is called the {Fubini-Study} metric, and it induces the \emph{standard measure} on $\HP^{N-1}$.

We say that $\cM$ is a \emph{projective algebraic manifold} if it is an open submanifold of the smooth part of a real projective variety $\cV\subset \HP^{N-1}$. We assume $\mathcal M$ to be $n$-dimensional, and consider a function $f\colon \mathcal M \to \mathbb R_{\geq 0}$ with a well-defined scaled probability density $f(x)/\int_\cM f(x)\d x$.
 For $A\in \HR^{n\times N}$, define the linear space $\cL_A=\{x\in \HP^{N-1}\mid Ax=0\}$ and write
\begin{equation*}
\overline f(A):=\sum_{x\in \cM\cap\cL_A} f(x).
\end{equation*}
In this section, we denote by $\varphi_\ell$ the density of the multivariate standard normal distribution on $\mathbb R^\ell$.
\begin{theorem}\label{theorem1_proj}In the notation introduced above:

(1) Let $f$ be an integrable function on $\mathcal M$. We have
\[
	\int_{\mathcal M} f(x)\,\d(x) = \mathrm{vol}(\HP^n)\, \mean_{A\sim \varphi_{n\times N}}\,\overline f(A).
\]

(2) Let $f$ be nonnegative and assume that the integral $\int_{\mathcal M} f(x)\,\d(x)$ is finite and nonzero. Let $X\in\mathcal M$ be the random variable obtained by choosing  $A\in \mathbb R^{n\times N}$ with probability
$\psi(A):=\frac{\varphi(A) \,\overline f(A)}{\mean_{A\sim \varphi_{n\times N}} \,\overline f(A)}$
and one of the finitely many points $X\in \mathcal M\cap \mathcal L_A$ with probability $f(x)/\overline f(A)$. Then $X$ is distributed accordding to the density $f(x)/\int_\cM f(x)\d x$.
\end{theorem}

\begin{remark}
In \cite[Section 2.4]{2017arXiv171103420L} Lairez proved a similar theorem for the uniform distribution on complex projective varieties.
\end{remark}

Sampling $\cL_A$ with $A\sim \varphi_{n\times N}$ yields a special distribution on the Grassmannian $\mathrm{G}(N-n-1, \HP^{N-1})$. By \cite{Leichtweiss} there is unique orthogonally invariant probability measure $\nu$
on $\mathrm{G}(N-n-1, \HP^{N-1})$. Since the distribution of the kernel of a Gaussian $A$ is invariant under orthogonal transformations, the projective plane $\cL_A=\{x\in\HP^{N-1}:Ax=0\}$  has distribution~$\nu$.

Furthermore, setting $f = 1$ in \cref{theorem1_proj} gives the formula
$$\mathrm{vol}(\cM)=\mathrm{vol}(\HP^n)\,\mean_{A\sim\varphi_{n\times N}} \,\lvert\mathcal M\cap\mathcal L_A \rvert.$$
This is the \emph{kinematic formula for projective manifolds} from \cite[Theorem 3.8]{howard} in disguise.

Before we can prove \cref{theorem1_proj}, we have to prove an auxiliary lemma, similar to \cref{alpha_eq}.

\begin{lemma}\label{alpha_eq_proj}
For any $x\in\cM$ we have
$$\int_{A\in\HR^{n\times N}: Ax=0}  |\det(A|_{\Tang{x}{\cM}})|\,\varphi_{n\times N}(A)\;\d A = \frac{1}{\mathrm{vol}(\HP^n)}.$$
In particular, the integral is independent of $x$.
\end{lemma}
\begin{proof}
Let $\cH(x):=\cset{A\in\HR^{n\times N}}{Ax = 0}$. It is a linear subspace of $\HR^{n\times N}$ of codimension~$n$. Let $U\in\HR^{N\times n}$ be a matrix whose columns form an orthonormal basis for~$\Tang{x}{{\cM}}$, so that $\det(\restr{A}{\Tang{x}{{\cM}}}) = \det(AU)$. Furthermore, let $O\in \HR^{N\times N}$ be an orthogonal matrix with $Ox = e_1$, where $e_1=(1,0,\ldots,0)^T\in\HR^N$. Then, $\cH(e_1) O= \cH(x)$. Making a change of variables $A\mapsto AO$  we get
\begin{equation}\label{eqqqq}
	\int_{A\in\cH(x)}  |\det(A|_{\Tang{x}{\cM}})|\,\varphi_{n\times N}(A)\,\d A=
	 \int_{A\in\cH(e_1)}\vert\det(AOU)\vert\, \varphi_{n\times N}(AO)\;\d A.
\end{equation}
We have $\varphi_{n\times N}(AO)=\varphi_{n\times N}(A)$, because the Gaussian distribution is orthogonally invariant. Moreover, any $A\in\cH(e_1)$ is of the form $A=[0, A']$ with $A'\in\HR^{n\times (N-1)}$, and we have $\varphi_{n\times N}(A)=\frac{1}{\sqrt{2\pi}^{\, n}}\varphi_{n\times (N-1)}(A')$. Let us denote by $O'$ the lower $(N-1)\times n$ part of $OU$, so that $AOU = A'O'$.
It follows that \cref{eqqqq} is equal to
\begin{equation*}
	\frac{1}{\sqrt{2\pi}^{\,n}}\int_{A'\in\HR^{n\times  (N-1)}}  \,\vert\det(A'O')\vert\,\varphi_{n\times (N-1)}(A')\;\d A'.
\end{equation*}
We show that $O'$ has orthonormal columns: since ${\cM}\subset \HS^{N-1}$, the tangent space $\Tang{x}{{\cM}}$ is orthogonal to $x$, which implies $U^T x = 0.$ Furthermore, $e_1^TOU = (U^T O^T e_1)^T = (U^T x)^T$. It follows that the first row of $OU$ contains only zeros and so the columns of $O'$ must be pairwise orthogonal and of norm one. A standard Gaussian matrix multiplied with a matrix with orthonormal columns is also standard Gaussian, so we have
\[
\int_{A'\in\HR^{n\times  (N-1)}}  \,\vert\det(A'O')\vert\,\varphi_{n\times (N-1)}(A')\;\d A'
=
\int_{M\in \mathbb R^{n\times n}} | \det(M) | \varphi_{n\times n}(M)\,\d M.
\]
This implies
$$\int_{A\in\cH(x)} \varphi_{n\times N}(A) |\det(A|_{\Tang{x}{\cM}})|\,\d A = \frac{\mean_{M\sim  \varphi_{n\times n}} \,\vert\det(M)\vert}{\sqrt{2\pi}^{\,n}}.$$
Finally, we compute $\mathrm{vol}(\HP^n) = \frac{1}{2}\mathrm{vol}(\HS^n) = \frac{\sqrt{\pi}^{\, n+1}}{\Gamma(\frac{n+1}{2})}$, and by \cite[Theorem 3.2.15]{muirhead}, we have $\mean_{M\sim  \varphi_{n\times n}}\det(M^TM)^\frac{1}{2} = \tfrac{1}{\sqrt{\pi}} \sqrt{2}^{\,n}\Gamma(\tfrac{n+1}{2})$. This finishes the proof.
\end{proof}

\begin{proof}[Proof of Theorem \ref{theorem1_proj}]
We define
${\cI} \coloneqq \cset{(A,x) \in  \HR^{n\times N}\times \cM}{A x = 0}.$ It is an algebraic subvariety of $\HR^{n\times N}\times \cM$. One can show that $\cI$ is smooth, but for our purposes it suffices to integrate over the dense subset of $\cI$ that is obtained by removing potential singularities from $\cI$.

Let~$\pi_1$ and $\pi_2$ be the projections from $ \cI$ to $\HR^{n\times N}$ and $ \cM$, respectively. Applying \cref{coarea} first to $\pi_1$ and then to $\pi_2$ yields
$$\mean_{A\sim \varphi_{n\times N}}\,\overline f(A) = \int_{\cM}f(x)\left( \int_{A\in\pi_1(\pi_2^{-1}(x))} \varphi_{n\times N}(A) \frac{\mathrm{NJ}(\pi_2, (A,x))}{\mathrm{NJ}(\pi_1, (A,x))}\,\d A \right) \d x.$$
By \cite[Sec.~13.2, Lemma~3]{BSS}, the ratio of normal Jacobians in the integrand equals $|\det(A|_{\Tang{x}{\cM}})|$. We get
\begin{align*}
\mean_{A\sim \varphi_{n\times N}}\,\overline f(A) &= \int_{\cM}f(x)\left( \int_{A\in\pi_1(\pi_2^{1}(x))} \varphi_{n\times N}(A) |\det(A|_{\Tang{x}{\cM}})|\,\d A \right) \d x\\
&= \frac{1}{\mathrm{vol}(\HP^n)} \,\int_{\cM}f(x)\d x;
\end{align*}
the second equality by \cref{alpha_eq_proj}. This proves the first part.

Let now $Y\in  \cI$ be the random variable obtained by choosing $A \in \HR^{n\times N}$ with distribution $\psi(A) = \frac{\varphi(A) \,\overline f(A)}{\mean_{A\sim \varphi_{n\times N}} \,\overline f(A)}$ and, independently of $A$, a point $ x\in \cM\cap {\cL}_A$ with probability
$ f(x) \overline f(A)^{-1}$. Then, by construction, $X =\pi_2(Y)$. Let $\gamma$ be the density of $X$.
Applying the first part of \cref{coarea-densities} to $\pi_1$ and then the second part to $\pi_2$, we have
\begin{align*}
\gamma(x)
&=\int_{(A,x)\in \pi_2^{-1}(x)}
\frac{\psi(A) f(x)\mathrm{NJ}(\pi_2, (A,x))}{\overline f(A)\mathrm{NJ}(\pi_1, (A,x))}\,\d A \\
&=\frac{ f(x)}{\mathbb E_\varphi (\overline f)}
\int_{(A,x)\in \pi_2^{-1}(x)}
\varphi(A)\,|\det(A|_{\Tang{x}{\cM}})|\,\d A\\
&=\frac{ f(x)}{\mathbb E_\varphi (\overline f)} \frac{1}{\mathrm{vol}(\HP^n)}\\
&=\frac{ f(x)}{\int_\cM f(x)\,\d x};
\end{align*}
the last penultimate equality again by \cref{alpha_eq_proj}, and the last equality by the first part of the theorem. This finishes the proof.
\end{proof}

\section{Comparison with previous work}\label{sec:comparison}
We now briefly review the use of kinematic formulae in applications and compare our method to them.

In~\cite{graph-cuts}, the authors use a Crofton-type formula for curves to establish a link between a discrete \emph{cut metric} on a grid, which is an object in combinatorial optimization, and an Euclidean metric on $\mathbb R^2$. This is then applied to a problem in image segmentation.

In~\cite{crofton-discretized}, the authors use the Cauchy formula and Crofton formula to compute Minkowski measures (e.g.\ surface area, perimeter) of discrete binary 2D or 3D pictures given as a grid of white or black pixels. Since the picture is discrete, the set of lines is also appropriately discretized, as well as Crofton's formula itself. So the difference to our method here lies in the discretization. In \cite{discretized-improved}, a more efficient way to evaluate the discretized Crofton's formula is proposed, using run-length encoding.

In \cite{crofton-sphere}, the authors use Crofton's formula to approximate the volume of a body $\cM$. They use a different sampling method, which goes as follows:
(1) Find a compact body~$\cE$ containing $\cM$, of known volume $\mathrm{vol}(\cE)$, such that the space of lines intersecting $\cE$ is approximately the same as the space of lines intersecting $\cM$. For example, $\cE$ could be a sphere containing $\cM$.
(2) Sample uniformly from the set of lines that intersect $\cE$.
(3) Compute the total number of intersection points of all the sampled lines with $\cE$ (call it~$g$) and with $\cM$ (call it $h$).
(4) Approximate the volume of $\cM$ as $\frac{h}{g}\,\mathrm{vol}(\cE)$.

As the authors write in~\cite{crofton-sphere}, this method can only give an approximation for the volume of~$\cM$. Its accuracy depends on the choice of $\cE$, i.e.\ on how well the uniform distribution on the set of lines intersecting $\cE$ approximates the same with respect to~$\cM$. On the other hand, our method is guaranteed to converge to the true volume given enough samples. We tested both methods on the curve $\cM$ from \cref{eq1} as well as on the ellipse
$\cM_1 = \{x\in \mathbb R^2\mid (x/3)^2 + y^2 - 1 = 0\}$, choosing $\cE$ to be the centered circle of radius $3$. We plotted the results in \cref{fig100}.

\begin{figure}[ht]
  \begin{center}
\includegraphics[height = 5cm]{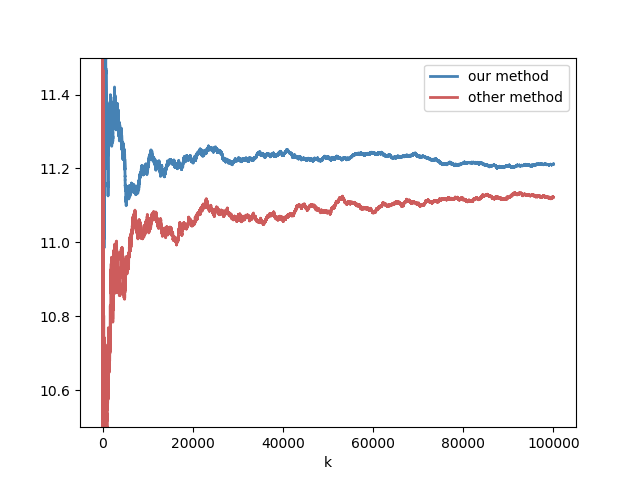}
\includegraphics[height = 5cm]{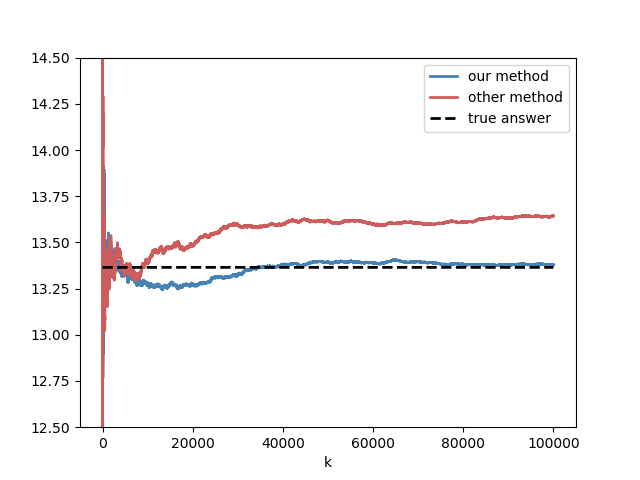}
\end{center}
\caption{The plot shows estimates for the volumes of two curves $\cM$ obtained from empirical estimates for $1\leq k\leq 10^5$ samples.
On the left, $\cM$ is the curve from \cref{eq1}. On the right, $\cM$ is the ellipse $\{x\in \mathbb R^2\mid (x/3)^2 + y^2 - 1 = 0\}$. Its volume is known and shown by the black line. The blue curve shows our method, and the red curve shows the method from~\cite{crofton-sphere}, where we have used the circle of radius $3$ for~$\cE$. Our method is guaranteed to converge to the true volume of $\cM$ for $k\to \infty$, while the other method is not, as exemplified by the plot. Our method seems to converge at least as quickly as the other method, if not slightly faster. \label{fig100}}
\end{figure}

Finally, we want to mention \cite{EM2016, Mangoubi2016}. In these works the authors derive MCMC methods for sampling $\mathcal M$ by intersecting it with random subspaces moving according to the \emph{kinematic measure} in $\mathbb R^N$. This is related to our discussion from the introduction, where we proposed sampling from $\psi(A,b)$ using MCMC methods. Taking this approach and comparing it to \cite{EM2016, Mangoubi2016} is left for future work as we discuss in the next section.

\section{Discussion}\label{sec:discussion}
We explained a new method to sample from a manifold $\mathcal M$ described by polynomial implicit equations. The implementation we described generates independent samples from the density $\psi(A,b)$ by rejection sampling, hence independent points $x\in \mathcal M$.
As we made experiments, we observed some downsides of our method. Namely, our method becomes slow when the degree of the variety is large in which case it is not easy to find a good $\kappa$, and the rejection rate in the sampling process becomes infeasibly large.

But we could also sample from $\psi(A,B)$ using an MCMC method with the goal of improving the rejection rate, at the cost of introducing dependencies between samples. In contrast to the known MCMC methods for nonlinear manifolds, our method would employ MCMC on a flat space. We name using MCMC methods for sampling $\psi(A,B)$ as a possible direction for future research.

\bibliographystyle{siamplain}
\bibliography{literature}

\begin{thebibliography}{10}

\bibitem{cyclo}
{\em juliahomotopycontinuation.org/examples/energy-minimization/}.

\bibitem{integration}
{\em juliahomotopycontinuation.org/examples/monte-carlo-integration/}.

\bibitem{sampling}
{\em juliahomotopycontinuation.org/examples/sampling/}.

\bibitem{bertini}
{\sc D.~Bates, J.~Hauenstein, A.~Sommese, and C.~Wampler}, {\em Bertini:
  Software for numerical algebraic geometry}.
\newblock Available at bertini.nd.edu with permanent doi:
  dx.doi.org/10.7274/R0H41PB5.

\bibitem{ripser}
{\sc U.~Bauer}, {\em Ripser: a lean {C}++ code for computation of
  {V}ietoris-{R}ips persistence barcodes}, https://github.com/Ripser/ripser.

\bibitem{sos}
{\sc G.~Blekherman, P.~Parrilo, and R.~Thomas}, {\em Semidefinite Optimization
  and Convex Algebraic Geometry}, vol.~13, MOS-SIAM Series on Optimization,
  2012.

\bibitem{BSS}
{\sc L.~Blum, F.~Cucker, M.~Shub, and S.~Smale}, {\em Complexity and real
  computation}, Springer, New York, 1998,
  \url{https://doi.org/10.1007/978-1-4612-0701-6},
  \url{http://dx.doi.org/10.1007/978-1-4612-0701-6}.

\bibitem{graph-cuts}
{\sc Y.~Boykov and V.~Kolmogorov}, {\em Computing geodesics and minimal
  surfaces via graph cuts}, in null, IEEE, 2003, p.~26.

\bibitem{BN2011}
{\sc C.~Brammer and D.~Nelson}, {\em Toward consistent terminology for
  cyclohexane conformers in introductory organic chemistry}, J. Chem. Educ.
  88(3),  (2011), pp.~292--294.

\bibitem{BT}
{\sc P.~Breiding and S.~Timme}, {\em Homotopy{C}ontinuation.jl - a package for
  solving systems of polynomial equations in {J}ulia}, Mathematical Software --
  ICMS 2018. Lecture Notes in Computer Science, 10931.
  juliahomotopycontinuation.org.

\bibitem{BB2010}
{\sc W.~H. Brown and L.~S. Brown}, {\em Organic Chemistry, Enhanced Edition},
  vol.~5, Brooks Cole, 2010.

\bibitem{Brubaker2012}
{\sc M.~Brubaker, M.~Salzmann, and R.~Urtasun}, {\em A family of {MCMC} methods
  on implicitly defined manifolds}, in Artificial intelligence and statistics,
  2012, pp.~161--172.

\bibitem{condition}
{\sc P.~B{\"u}rgisser and F.~Cucker}, {\em {Condition: The Geometry of
  Numerical Algorithms}}, vol.~349 of Grundlehren der mathematischen
  Wissenschaften, Springer, Heidelberg, 2013.

\bibitem{BG2013}
{\sc S.~Byrne and M.~Girolami}, {\em Geodesic monte carlo on embedded
  manifolds}, Scandinavian J. Stat., 41 (2013).

\bibitem{CG2011}
{\sc B.~Calderhead and M.~Girolami}, {\em Riemann manifold langevin and
  hamiltonian monte carlo methods}, J. Royal Statistical Society B, 73 (2011).

\bibitem{MCSM}
{\sc G.~Casella and C.~Robert}, {\em Monte Carlo Statistical Methods}, Springer
  texts in statistics, 2004.

\bibitem{HOM4PS}
{\sc T.~Chen, T.~Lee, T.~Li, and N.~Ovenhouse}, {\em {HOM4PS}: a software
  package for solving polynomial systems by the polyhedral homotopy
  continuation method}.
\newblock Available at \url{hom4ps3.org}.

\bibitem{sampling_hauenstein}
{\sc E.~{Dufresne}, P.~B. {Edwards}, H.~A. {Harrington}, and J.~D.
  {Hauenstein}}, {\em {Sampling real algebraic varieties for topological data
  analysis}}, ArXiv e-prints,  (2018), \url{https://arxiv.org/abs/1802.07716}.

\bibitem{EM2016}
{\sc A.~Edelman and O.~Mangoubi}, {\em Integral geometry for markov chain monte
  carlo: overcoming the curse of search-subspace dimensionality},  (2015),
  \url{https://arxiv.org/abs/1503.03626}.

\bibitem{parker}
{\sc P.~Edwards}, {\em Personal communication}.

\bibitem{FS2007}
{\sc M.~Farber and F.~Sch\"utz}, {\em Homology of planar polygon spaces},
  Geometriae Dedicata 125(1),  (2007), pp.~75--92.

\bibitem{GHZ2017}
{\sc H.-C.~M. Goodman, J. and E.~Zappa}, {\em Monte carlo on manifolds:
  Sampling densities and integrating functions}, Comm. Pure Appl. Math., 2
  (2017).

\bibitem{Harris1992}
{\sc J.~Harris}, {\em {Algebraic Geometry, A First Course}}, vol.~133 of
  Graduate Text in Mathematics, Springer-Verlag, 1992.

\bibitem{howard}
{\sc R.~Howard}, {\em The kinematic formula in {R}iemannian homogeneous
  spaces}, Mem. Amer. Math. Soc., 106 (1993), pp.~vi+69,
  \url{https://doi.org/10.1090/memo/0509},
  \url{http://dx.doi.org/10.1090/memo/0509}.

\bibitem{Hunter:2007}
{\sc J.~D. Hunter}, {\em Matplotlib: A 2d graphics environment}, Computing In
  Science \& Engineering, 9 (2007), pp.~90--95,
  \url{https://doi.org/10.1109/MCSE.2007.55}.

\bibitem{Kalos2008}
{\sc M.~H. Kalos and P.~A. Whitlock}, {\em Monte {C}arlo methods},
  Wiley-Blackwell, Weinheim, second~ed., 2008,
  \url{https://doi.org/10.1002/9783527626212},
  \url{https://doi.org/10.1002/9783527626212}.

\bibitem{2017arXiv171103420L}
{\sc P.~{Lairez}}, {\em {Rigid continuation paths I. Quasilinear average
  complexity for solving polynomial systems}}, ArXiv e-prints,  (2017),
  \url{https://arxiv.org/abs/1711.03420}.

\bibitem{crofton-discretized}
{\sc D.~Legland, K.~Ki{\^e}u, and M.-F. Devaux}, {\em Computation of minkowski
  measures on 2d and 3d binary images}, Image Analysis \& Stereology, 26
  (2007), pp.~83--92.

\bibitem{discretized-improved}
{\sc G.~Lehmann and D.~Legland}, {\em Efficient n-dimensional surface
  estimation using crofton formula and run-length encoding}, Efficient
  N-Dimensional surface estimation using Crofton formula and run-length
  encoding, Kitware INC (2012),  (2012).

\bibitem{Leichtweiss}
{\sc K.~Leichtweiss}, {\em {Zur Riemannschen Geometrie in Grassmannschen
  Mannigfaltigkeiten}}, Mathematische Zeitschrift, 76 (1961).

\bibitem{LRS2018}
{\sc R.-M. {Leli{\`e}vre}, T. and G.~{Stoltz}}, {\em Hybrid monte carlo methods
  for sampling probability measures on submanifolds}, Numer. Math.,  (2019).

\bibitem{free_energy}
{\sc T.~Leli{\`e}vre, M.~Rousset, and G.~Stoltz}, {\em Free Energy
  Computations: A Mathematical Perspective}, Imperial College Press, 2010.

\bibitem{Lelievre2012}
{\sc T.~Leli\`evre, M.~Rousset, and G.~Stoltz}, {\em Langevin dynamics with
  constraints and computation of free energy differences}, Math. Comp., 81
  (2012), pp.~2071--2125,
  \url{https://doi.org/10.1090/S0025-5718-2012-02594-4},
  \url{https://doi.org/10.1090/S0025-5718-2012-02594-4}.

\bibitem{Leykin2018}
{\sc A.~Leykin}, {\em Homotopy continuation in macaulay2}.
\newblock Mathematical Software -- ICMS 2018. Lecture Notes in Computer
  Science.

\bibitem{crofton-sphere}
{\sc X.~Li, W.~Wang, R.~R. Martin, and A.~Bowyer}, {\em Using low-discrepancy
  sequences and the crofton formula to compute surface areas of geometric
  models}, Computer-Aided Design, 35 (2003), pp.~771--782.

\bibitem{Mangoubi2016}
{\sc O.~Mangoubi}, {\em Integral geometry, Hamiltonian dynamics, and Markov
  Chain Monte Carlo, PhD Thesis}, Massachusetts Institute of Technology, 2016.

\bibitem{muirhead}
{\sc R.~Muirhead}, {\em Aspects of Multivariate Statistical Theory}, vol.~131,
  J. Wiley \& Sons, NY, 1982.

\bibitem{santalo}
{\sc L.~Santal\'o}, {\em Integral geometry and geometric probability}, vol.~1,
  Addison-Wesley, 1976.

\bibitem{SG2003}
{\sc J.~{Skowron} and A.~{Gould}}, {\em {General Complex Polynomial Root Solver
  and Its Further Optimization for Binary Microlenses}}, arXiv:1203.1034.

\bibitem{PHCpack}
{\sc J.~Verschelde}, {\em {PHC}pack: a general-purpose solver for polynomial
  systems by homotopy continuation}.
\newblock phcpack.org.

\bibitem{whitney}
{\sc H.~Whitney}, {\em Elementary structure of real algebraic varieties.},
  Annals of Math., 66 (1957).

\end{thebibliography}
\end{document}